\documentclass{amsproc}
\usepackage{amsmath,amsfonts}

\newtheorem{thm}{Theorem}[section]
\newtheorem{lem}[thm]{Lemma}
\newtheorem{cor}[thm]{Corollary}
\newtheorem{prop}[thm]{Proposition}

\begin{document}

\def\l{\lambda}
\def\t{\theta}
\def\T{\Theta}
\def\m{\mu}
\def\a{\alpha}
\def\b{\beta}
\def\g{\gamma}
\def\o{\omega}
\def\p{\varphi}
\def\D{\Delta}
\def\O{\Omega}
\def\G{\Gamma}

\def\N{{\mathbb N}}
\def\C{{\mathbb C}}
\def\Z{{\mathbb Z}}
\def\R{{\mathbb R}}
\def\P{{\mathbb P}}
\def\E{{\mathbb E}}
\def\K{{\mathbb K}}

\def\L{{\mathcal L}}
\def\D{{\mathcal D}}
\def\I{{\mathcal I}}
\def\T{{\mathcal T}}
\def\M{{\mathcal M}}

\title{Interlaced processes on the circle}

\author{Anthony P. Metcalfe}
\address{Department of Mathematics, University College Cork, Ireland.}
\email{am12@cs.ucc.ie}
\author{Neil O'Connell}
\address{Mathematics Institute, University of Warwick, Coventry CV4 7AL, UK.}
\email{n.m.o-connell@warwick.ac.uk}
\author{Jon Warren}
\address{Department of Statistics, University of Warwick, Coventry CV4 7AL, UK.}
\email{j.warren@warwick.ac.uk}

\begin{abstract}
When two Markov operators commute, it suggests that we can couple
two copies of one of the corresponding processes. 
We explicitly construct a number of couplings of this type for a commuting
family of Markov processes on the set of conjugacy classes of the unitary group, 
using a dynamical rule inspired by the RSK algorithm.
Our motivation for doing this is to develop a parallel programme,
on the circle, to some recently discovered connections in random matrix theory
between reflected and conditioned systems of particles on the line.
One of the Markov chains we consider gives rise to a family 
of Gibbs measures on `bead configurations' on the infinite cylinder. 
We show that these measures have determinantal 
structure and compute the corresponding space-time correlation kernel.
\end{abstract}

\maketitle

\section{Introduction}

When two Markov operators commute, it suggests that we can couple
two copies of one of the corresponding processes.
Such couplings have
been described in~\cite{df} in a general context. 
In this paper, we explicitly construct a number of couplings of this type for a commuting
family of Markov processes on the set of conjugacy classes of the unitary group, 
using a dynamical rule inspired by the Robinson-Schensted-Knuth (RSK) algorithm.
Our motivation for doing this is to develop a parallel programme,
on the circle, to some recently discovered connections in random matrix theory
between reflected and conditioned systems of particles on the line
(see, for example, \cite{bar01,bbo05,gtw01,oy02,war06}).
The RSK algorithm is a combinatorial device which plays an important role in the 
representation theory of $GL(n,\C)$ and lies at the heart of these developments.  
We refer the reader to~\cite{fulton} for more background on the combinatorics.  
One of the Markov chains we consider gives rise to a family 
of Gibbs measures on `bead configurations' on the infinite cylinder. 
This is related to recent work~\cite{bou06,bt07,kos06}
on planar and toroidal models. We will show that these measures have determinantal 
structure and compute the corresponding space-time correlation kernel.

We start with some motivation, and a flavour of some of the main results in this paper.
The following construction is closely related to the RSK algorithm~\cite{oy02,noc03,jphys}.  
Let $E_1=\Z$ and, for $n\ge 2$,
$$E_n=\{x\in\Z^n:\ x_1<\cdots<x_n\}.$$
The reader should think of an element $x\in E_n$ as a configuration of $n$ particles
on the integers located at positions $x_1<\cdots <x_n$.
We say that a pair $(x,y)\in E_m\times E_{m+1}$ are {\em interlaced},
and write $x\preceq y$, if $y_j<x_j\le y_{j+1}$ for all $j\le m$.
A (discrete) Gelfand-Tsetlin pattern of depth $n$ is a collection $(x^1,x^2,\ldots,x^n)$
such that $x^m\in E_m$ for $m\le n$ and $x^m\preceq x^{m+1}$ for $1\le m<n$.

Fix $n\ge 2$ and denote by $GT_n$ the set of Gelfand-Tsetlin patterns of depth $n$.
Let $w_1,w_2,\ldots$ be a sequence of independent random variables, each chosen 
according to the uniform distribution on $\{1,2,\ldots,n\}$.  Using these, we will construct 
a Markov chain $(X(k),k\ge 0)$ with state space $GT_n$, which evolves according to
$$X(k+1)=g(X(k),w_{k+1}),\qquad k\ge 0$$
where $g:GT_n\times\{1,\ldots,n\}\to GT_n$ is defined recursively as follows.
Fix $m<n$ and let $(x,y)\in E_m\times E_{m+1}$ such that $x\preceq y$.
Let $x'\in E_m$ such that, for some $j\le m$, $x'_i=x_i+\delta_{ij}$. 
The reader should have in mind $m$ particles located at positions $x_1<\cdots<x_m$,
interlaced with another set of $m+1$ particles located at positions $y_1<\cdots<y_{m+1}$.
The first configuration $x$ is updated by moving the particle at position $x_j$ one step to the right,
that is, to position $x_j+1$, giving a new configuration $x'$.  This can be used to obtain an 
update $y'$ to the second configuration $y$, obtained by moving the first available particle, strictly
to the right of position $y_j$, which can be moved {\em without breaking the interlacing constraint}, 
so that $x'\preceq y'$.  Such a particle is guaranteed to exist because the interlacing constraint 
cannot be broken by the rightmost particle. In other words, the updated configuration is given 
by $y'_i=y_i+\delta_{ik}$ where $k=\inf\{l>j:\ y_l+1<x_l\}$.  Let us write $y'=\phi(x,y,x')$
where $\phi$ is defined on an appropriate domain.  Now, given $m\le n$ and a pattern
$(x^1,\ldots,x^n)\in GT_n$ we define a new pattern 
$$(y^1,\ldots,y^n) \equiv g((x^1,\ldots,x^n),m)$$ 
as follows.
First we set $y^j=x^j$ for $j<m$. Then we obtain $y^m$ from the configuration $x^m$ by 
moving the first available particle, starting from the particle at position $x^m_1$, which 
can be moved one step to the right without breaking the interlacing constraint. Finally,
we define, for $m\le l < n$, $y^{l+1}=\phi(x^l,x^{l+1},y^l)$.

The Markov chain $X$ has remarkable properties.  For example, if we start from
the initial pattern 
\begin{equation}\label{null}
X(0)=((0),(-1,0),(-2,-1,0),\ldots,(-n+1,\ldots,0)),
\end{equation}
then $(X^n(k),k\ge 0)$ is a Markov chain (with respect to its own filtration)
with state space $E_n$ and transition probabilities given by
\begin{equation}
P_n(x,y)=\begin{cases} \frac1n\frac{h(y)}{h(x)} & x\nearrow y\\
0 & \mbox{otherwise} \end{cases}
\end{equation}
where $h$ is the Vandermonde function $h(x)=\prod_{i<j\le n} (x_j-x_i)$
and the notation $x\nearrow y$ means that the configuration $y$ can be obtained from 
$x$ by moving one particle one step to the right.  The fact that $P_n$ is a Markov
transition kernel follows from the fact (see, for example,~\cite{kor02}) that $h$ is
positive on $E_n$ and satisfies
$$\frac1n\sum_{x\nearrow y} h(y) = h(x).$$
More generally (see, for example,~\cite{noc03,jphys}):
\begin{prop}\label{rsk} 
Given $x\in E_n$, if $X(0)$ is
chosen uniformly at random from the set of patterns $(x^1,\ldots,x^n)$ with $x^n=x$,
then $(X^n(k),k\ge 0)$ is a Markov chain started from $x$ with transition probabilities
given by $P_n$.  Moreover, for each $T>0$, the conditional law of $X(T)$, given 
$(X^n(k),T\ge k\ge 0)$, is uniformly distributed on the set of patterns $(x^1,\ldots,x^n)$ 
with $x^n=X^n(T)$.
\end{prop}
The connection with the RSK algorithm is the following: 
if we start from the initial pattern (\ref{null}), then the integer partition 
$$(X_n^n(k),X^n_{n-1}(k)+1,X^n_{n-2}(k)+2,\ldots,X^n_{1}(k)+n-1)$$
is precisely the shape of the tableau obtained when one applies the RSK algorithm,
with column insertion, to the word $w_1w_2\cdots w_k$.  We refer to~\cite{noc03}
for more details.  This construction clearly has a nested structure.  If we consider
the evolution of the last two rows $X^{n-1}$ and $X^n$ we see that we can construct
a Markov chain with transition probabilities $P_n$ from a Markov chain with transition
probabilities $P_{n-1}$ plus a `little bit of extra randomness'. 

In the above construction we are thinking in terms of particles moving on a line.  
In random matrix theory, there are often strong parallels between natural measures
on configurations of particles (or `eigenvalues') on the line, and configurations of
particles on the circle.  It is therefore natural to ask if there is an analogue of the above 
construction for configurations of particles on the circle.  The notion of interlacing carries 
over in the obvious way.  However, in this setting, interlaced configurations should have 
the {\em same} number of particles, and the analogue of a Gelfand-Tsetlin pattern
could potentially be an infinite object. Despite this fundamental difference between
the two settings, there is indeed a natural analogue of the above `RSK dynamics' 
and a natural analogue of Proposition~\ref{rsk}.  Consider the discrete circle with
$N$ positions which we label $\{0,1,\ldots,N-1\}$ in the anti-clockwise direction.
The analogue of the Markov chain with transition matrix $P_n$ is a Markov
chain on the set $C_n^N$ of configurations of $n$ indistinguishable particles on 
the discrete circle with transition probabilities given by
\begin{equation}
Q(x,y)=\begin{cases} \frac cn\frac{\Delta(y)}{\Delta(x)} & x\nearrow y\\
0 & \mbox{otherwise} \end{cases}
\end{equation}
where, similarly as before, the notation $x\nearrow y$ means that the configuration 
$y$ can be obtained from $x$ by moving one particle one step anti-clockwise,
and the function $\Delta$ is again a Vandermonde function defined, for a configuration
$x$ which consists of a particles located at positions $k_1,\ldots,k_n$, by 
\begin{equation}\label{pf}
\Delta(x)=\prod_{i<j\le n} |e^{2\pi i k_j/N}-e^{2\pi i k_i/N}|.
\end{equation}
The constant $c$ is chosen so that $Q(x,\cdot)$ is a probability distribution;
the fact that $c$ can be chosen independently of $x$ follows from the fact
(see, for example,~\cite{kor02}) that $\Delta$ is a positive eigenvector of the
matrix $1_{x\nearrow y}$, that is,
$$\sum_{x\nearrow y} \Delta(y)=\lambda \Delta(x),$$
for some $\lambda>0$.  A formula for the eigenvalue $\lambda$ can be found in~\cite{kor02}.

In this setting we will say that a pair of configurations $(x,y)$ are interlaced, 
and write $x\preceq y$ as before, if there is a labelling of the particles
such that $x$ consists of a particles located at positions $k_1,\ldots,k_n$,
$y$ consists of a particles located at positions $l_1,\ldots,l_n$,
$k_j<l_j\le k_{j+1}$ for $j<n$ and either $k_n<l_n\le k_1+N$ or $k_n<l_n+N\le k_1+N$.

Let $(X(k),k\ge 0)$ be a Markov chain with state space $C_n^N$ and transition matrix $Q$.
On the same probability space, without using any extra randomness, we can construct a 
second process $(Y(k),k\ge 0)$, also taking values in $C_n^N$, such that $X(k)\preceq Y(k)$ 
for all $k$.  This is given as follows.  For each $k>0$, given $X(k)$, $Y(k)$ and $X(k+1)$
we obtain the configuration $Y(k+1)$ from $Y(k)$ by moving a particle one step anti-clockwise;
this particle is chosen as follows. The transition from $X(k)$ to $X(k+1)$ involves one
particle moving one step anti-clockwise; starting at the position of this particle, choose
the first particle in the configuration $Y(k)$ that we come to, in an anti-clockwise direction,
which can be moved by one step anti-clockwise without breaking the interlacing constraint.

The function $\Delta$ defined by (\ref{pf}) is also a positive (left and right) eigenvector of the 
matrix $1_{x\preceq y}$.  This follows, for example, from the discussion given at the end of Section 3.
In particular,
$$\sum_{x\preceq y} \Delta(x) = \gamma \Delta(y)$$
for some $\gamma>0$ and we can define a Markov kernel on $C_n^N$ by
\begin{equation}\label{m}
M(y,x)=\frac1\gamma\frac{\Delta(x)}{\Delta(y)}1_{x\preceq y}.
\end{equation}
The analogue of Proposition~\ref{rsk} in this setting is the following:
\begin{prop}\label{crsk}
If $Y(0)=y$ and $X(0)$ is chosen at random according to the distribution $M(y,\cdot)$, then
$(Y(k),k\ge 0)$ is a Markov chain started at $y$ with transition matrix $Q$.  Moreover, for
each $T>0$, the conditional law of $X(T)$, given $(Y(k),T\ge k\ge 0)$, is given by $M(Y(T),\cdot)$.
\end{prop}
In the sequel we will present a number of variations of this result, firstly involving random
walks with jumps in a continuous state space and secondly involving Brownian motion.
Proposition~\ref{crsk} follows from results presented in Section 3 (see discussion towards
the end of that section).

We will also study continuous analogues of the Markov chain with transition matrix $M$.
These Markov chains also arise naturally in the context of a certain random walk on the 
unitary group which is obtained by taking products of certain random (complex) reflections, 
as studied for example in~\cite{fr06,porod2}.  This is described in Section 2 and taken as
a starting point for the exposition that follows.  The main point is that these Markov
chains commute with each other and with the Dirichlet Laplacian on the set of conjugacy
classes of the unitary group.  In Sections 3 and~4, we present couplings which realise these 
commutation relations, first between interlaced random walks on the circle and later between 
interlaced Brownian motions.  These couplings are precisely the variations of Proposition~\ref{crsk}
mentioned above.  Actually there are two natural couplings for the random walks and these 
correspond to dynamical rules inspired by the RSK algorithm with row, and column insertion, 
respectively.  The couplings between interlaced Brownian motions can be thought of as a 
limiting case where the two types of coupling become equivalent. In Section 5, we consider a 
family of Markov chains which can be thought of as perurbations of a continuous analogue 
of the Markov chain with transition matrix given by~(\ref{m}).  These give rise to a natural 
family of Gibbs measures 
on `bead configurations' on the infinite cylinder. This is a cylindrical analogue of the planar bead 
model studied in [5] and, in one special case (the `unperturbed' case), can also be 
regarded as a cylindrical analogue of some natural measures on Gelfand-Tsetlin patterns 
related to `GUE minors'~\cite{bar01,jn,or,fn} (see also~\cite{manon} for extensions to the other 
classical complex Lie algebras).  We show that these measures have determinantal
structure by first writing the restrictions of these measures to cylinder sets
as products of determinants and then following the methodology of Johansson
(see, for example,~\cite{j}) to compute the space-time correlation functions.  

\section{Markov processes on the conjugacy classes of the unitary group}

Consider the group $U(n)$ of $n\times n$ unitary matrices, and denote by $C_n$ 
the set of conjugacy classes in $U(n)$. Each element of $C_n$ can be identified
with an unlabelled configuration of $n$ points (eigenvalues) on the unit circle, which in turn
can be identified with the Euclidean set $D_n=\{\t\in\R^n:\ 0\le \t_1\le \cdots\le \t_n<2\pi\}$.
Denote by $dx$ the image of Lebesgue measure under the latter identification,
and by $\mu$ the probability measure on $C_n$ induced from Haar measure on $U(n)$.
Then $\mu(dx)=(2\pi)^{-n}\Delta(x)^2dx$, where 
$\Delta(x)$ is defined, for $x=\{e^{i\t_1},\ldots,e^{i\t_n}\}$, by
\begin{equation}\label{Delta}
\Delta(x)=\prod_{1\le l<m\le n} | e^{i\t_l}-e^{i\t_m}|.
\end{equation}
The irreducible characters $\chi_\l$ of $U(n)$ are indexed by the set 
$$\O_n=\{\l\in\Z^n:\ \l_1\ge\l_2\ge\cdots \ge\l_n\}.$$
We assume that these are normalised, that is, $\int |\chi_\l|^2d\mu=1$ for all $\l\in\O_n$. 

Let $x\in C_n$ and consider the random walk on $U(n)$ which is constructed
by multiplying together independent, randomly chosen elements from the conjugacy 
class $x$. (By `a randomly chosen element from the conjugacy class $x$' we mean
a random element of the conjugacy class $x$ which admits a representation of the form 
$MD_xM^*$, where $D_x$ is an arbitrary element of the conjugacy class $x$ and 
$M$ is a Haar-distributed, randomly chosen element of $U(n)$.) 
The corresponding Markov kernel $p_x(y,dz)$ can be 
interpreted as the law of $D_xMD_yM^*$, where $D_x$ and $D_y$ are arbitrary elements of 
the conjugacy classes $x$ and $y$, respectively, and $M$ is a Haar-distributed randomly
chosen element of $U(n)$.
Write $p_xf=\int p(\cdot,dz)f(z)$ for $f\in L_2(C_n,\mu)$.
Then, for each $\l\in\O_n$,
\begin{equation}\label{efun}
p_x \chi_\l =\frac{\chi_\l(x)}{d_\l} \chi_\l ,
\end{equation}
where $d_\l$ is the dimension of the representation corresponding to $\l$
(see, for example,~\cite[Proposition 6.5.2]{faraut}).
Moreover, $\mu$ is an invariant measure for the Markov kernel
$p_x$, that is,
\begin{equation}
 \int \mu(dy)p_x(y,\cdot ) = \mu.
\end{equation}

The Markov kernels $\{p_x,\ x\in C_n\}$ are the extreme points in the convex set 
$\M$ of all Markov kernels on $L_2(C_n,\mu)$ with the irreducible characters as 
eigenfunctions. (See, for example,~\cite{bh}).) All of the Markov kernels in $\M$ have
$\mu$ as an invariant measure and, as operators on $L_2(C_n,\mu)$, they commute.
They correspond to random walks on $U(n)$ such that the law of the increments is 
invariant under conjugation. Such random walks
on $U(n)$, and other classical compact groups, have been studied extensively
in the literature (see, for example,~\cite{ds,diaconis,rosenthal,porod1,porod2}).

The case of interest in this paper is the random walk obtained by taking products of 
certain random (complex) reflections in $U(n)$ (see, for example,~\cite{fr06,porod2}). 
More precisely, we take $x=\{e^{ir},1,1,\ldots,1\}$ in the above kernel, where $r\in(0,2\pi)$.  
Let us write 
\begin{equation}\label{pr-def}
p_r:=p_x
\end{equation} for this case.  A concrete description of this kernel can be given 
as follows (see~\cite{fr06} for details). For $a,b\in D_n$, write $a\preceq b$ if 
$$a_1\le b_1\le a_2 \le \cdots \le a_n \le b_n.$$
For $y=\{e^{ia_1},\ldots,e^{ia_n}\}$ and $z=\{e^{ib_1},\ldots,e^{ib_n}\}$, 
where $a,b\in D_n$, write $y\preceq_r z$ if either 
$a\preceq b$ and $\sum_j (b_j-a_j) =r$, or 
$b\preceq a$ and $\sum_j (b_j-a_j)+2\pi =r$.
The measure $p_r(y,\cdot)$ is supported on the set 
$$F_r(y)=\{z\in C_n:\ y\preceq_r z\}.$$ This set can be identified
with the disjoint union of a pair of $(n-1)$-dimensional Euclidean sets 
$$\{ b\in D_n:\ a\preceq b ,\  \sum_j (b_j-a_j) =r\}\cup\{b\in D_n:\ a\succeq b,\ \sum_j (b_j-a_j) =r-2\pi\},$$
each of which can be endowed with $(n-1)$-dimensional
Lebesgue measure giving a natural measure on their union.
The measure obtained on $F_r(y)$ via this identification can be extended to a measure
$\nu_r(y,\cdot)$ on $C_n$ by setting $\nu_r(y,C_n \backslash F_r(y))=0$.
The following identity can be deduced from~\cite[Lemma 2]{fr06}.
\begin{prop}\label{br}
\begin{equation}\label{pr}
p_r(y,dz)= \frac{1}{\gamma_r}  \frac{\Delta(z)}{\Delta(y)} \nu_r(y,dz) ,
\end{equation}
where $\gamma_r= |1-e^{ir}|^{(n-1)}/(n-1)!$.
\end{prop}
In Section 5, we will present a determinantal formula for $\int_0^{2\pi}  q^r \nu_r(y,dz)dr$,
where $q>0$ is a parameter, and use this to give an alternative proof of Proposition~\ref{br}.

Another operator which will play a role in this paper is the Dirichlet Laplacian 
on the (closed) alcove
\begin{equation}\label{alcove}
A_n=\{\t\in\R^n:\ \t_1\le\t_2\le\cdots\le\t_n\le\t_1+2\pi\}.
\end{equation}
Let $(Q_t)$ denote the transition semigroup of a standard Brownian motion 
conditioned never to exit $A_n$. This is a Doob transform of the Brownian
motion which is killed when it exits $A_n$, via the positive eigenfunction
\begin{equation}\label{h}
h(\t)=\prod_{1\le l<m\le n} | e^{i\t_l}-e^{i\t_m}|.
\end{equation}
We can identify $C_n$ with $\exp(iA_n)/\mathfrak{C}_n$,
where $\mathfrak{C}_n$ denotes the group of cyclic permutations
which acts on $A_n$ by permuting coordinates.  It is known~\cite{berard} 
that the eigenfunctions
of the induced semigroup $(\hat Q_t)$ on $L_2(C_n,\mu)$ are given by 
the irreducible characters $\{\chi_\l,\ \l\in\O_n\}$, which implies that
$\hat Q_t\in\M$, for each $t>0$. Note that the corresponding process
on $C_n$ can be thought of as $n$ standard Brownian motions on the circle
conditioned never to collide.

We will also consider discrete analogues of the above processes. Set 
$$A_n^N=(N A_n/2\pi)\cap\Z^n,\quad C_n^N=\exp(2\pi iA_n(N)/N)/\mathfrak{C}_n,
\quad\O_n^N=\{\l\in\O_n:\ \l_1\le N-1\}.$$
Much of the above discussion can be replicated in this setting, but for our purposes it
suffices to make the following remark. Think of $C_n^N$ as the set of configurations
of $n$ particles at distinct locations on the discrete circle with $N$ positions.
Consider the random walk in $C_n^N$ where at each step a particle is chosen at 
random and moved one position anti-clockwise if that position is vacant; if it is
not vacant the process is killed. It is known that the restriction of the function 
$\Delta$ to $C_n^N$ is the Perron-Frobenius eigenfunction for this sub-Markov
chain. In fact, a complete set of eigenfunctions (with respect to the measure $\Delta(x)^2$) 
is given by the restrictions of the characters $\{\chi_\l,\ \l\in\O_n^N\}$ to $C_n^N$.
(See, for example,~\cite{kor02}.) As we shall see later, the discrete analogues of the
$\nu_r$ (thought of as operators) commute with the transition kernel of this killed 
random walk and therefore share these eigenfunctions.

\section{Couplings of interlaced  random walks}

The Markov chain on $C_n$ with transition probabilities $p_r$,
defined by (\ref{pr-def}), can be lifted to a Markov chain
on $A_n$, which is better suited to the constructions of this section. To 
make this  precise let us 
say $x$ and $x^\prime$  belonging to $A_n$ are $r$-interlaced, for some $r 
\in (0,2\pi)$, if 
\begin{equation}
x_i^\prime\in [x_{i},x_{i+1}] \text{ for } i=1,2,\ldots ,n  \text{ and } 
\sum_{i=1}^n (x^\prime_i-x_i)=r,
\end{equation}
when we adopt the convention that $x_{n+1}=x_1+2\pi$.
In this case we will write $x \preceq_r x^\prime $. 
Define $\pi: A_n \rightarrow C_n$ by $\pi(x)=\{e^{ix_1}, \ldots, e^{ix_n}\}$.
Denote  by $l_r(x,dx^\prime)$ the $(n-1)$-dimensional Lebesgue measure on 
the set $$G_r(x)=\{ x^\prime \in A_n: x \preceq_r x^\prime \} .$$
Clearly the restriction of $\pi$ to $G_r(x)$ is injective, with $\pi(G_r(x))=F_r(\pi(x))$.
Moreover, for measurable $B\subset A_n$, $l_r(x,B)= \nu_r(\pi(x),\pi(B\cap G_r(x)))$.
Thus, if we define, for measurable $B\subset A_n$,
\begin{equation}\label{q}
q_r(x,B)= p_r(\pi(x),\pi(B\cap G_r(x)))
\end{equation} then, by Proposition \ref{br}, 
\begin{eqnarray*}
q_r(x,B) &=& \int_{\pi(B\cap G_r(x))} p_r(\pi(x),dz) \\
&=& \int_{\pi(B\cap G_r(x))} \frac{1}{\gamma_r} \frac{\Delta(z)}{\Delta(\pi(x))} \nu_r(\pi(x),dz)\\
&=& \int_B \frac{1}{\gamma_r} \frac{\Delta(\pi(x^\prime))}{\Delta(\pi(x))} l_r(x,dx^\prime),
\end{eqnarray*}
and hence,
\begin{equation}
q_r(x,dx^\prime)= \frac{1}{\gamma_r}
\frac{\Delta(\pi(x^\prime))}{\Delta(\pi(x))} l_r(x,dx^\prime).
\end{equation}
We will refer to a Markov chain with values in $A_n$ and transition 
probabilities $q_r$ as an $r$-interlacing random walk.  As far as we
are aware, such processes have not previously appeared in the literature.
Note that since $p_rp_s=p_sp_r$ we have $q_rq_s=q_sq_r$ or, equivalently, 
$l_rl_s=l_sl_r$.

The goal of this section is to construct, for  given $r,s \in (0,2\pi)$,  
two different Markovian couplings   $\bigl(X(k), Y(k); k \geq 0 )$ 
of a pair of $r$-interlacing random walks on $A_n$, having the 
property that $X(k)$ and $Y(k)$ are $s$-interlaced for each $n$ and 
moreover, for each $l \geq 0$, the trajectory  $\bigl(Y(k); 0 \leq k \leq 
l )$ 
will be a deterministic function of 
$Y(0)$ together with the trajectory $\bigr(X(k); 0 \leq k \leq l\bigr)$.

The existence of such  couplings is suggested by the commutation relation 
$q_r q_s=q_s q_r$, equivalently $l_rl_s=l_sl_r$.
For any $u,v \in A_n$   consider the two sets $\tau_{u,v}=\{ x \in A_n: u 
\preceq_ s x \preceq_r  v\}$ and $\tau^\prime_{u,v}=\{y \in A_n : u 
\preceq_r y 
\preceq_s v \}$. If either is non-empty, then they both are, and in this 
case they are $(n-1)$-dimensional polygons, and  the  relation 
$l_rl_s=l_sl_r$  can be interpreted as saying these two polygons have the 
same $(n-1)$-dimensional volume. In fact the two polygons are congruent. 
Define 
\begin{equation}\label{phi}
y=\phi_{u,v}(x)
\end{equation} via
\begin{equation}
y_i=  \min(u_{i+1},v_i)+\max (u_i,v_{i-1}) -x_i.
\end{equation}
It is easy to see that  $\phi_{u,v}$ is an isometry from $\tau_{u,v}$ to 
$\tau_{u,v}^\prime$ using the facts that
\[
y_i \in  [\max(u_{i},v_{i-1}),\min (u_{i+1},v_{i}) ] \text{ if and only if 
}  x_i \in  [\max(u_{i},v_{i-1}),\min (u_{i+1},v_{i}) ], 
\]
and 
\[
\sum_{i=1}^n y_i= \sum_{i=1}^n  (\min(u_{i+1},v_i)+\max 
(u_i,v_{i-1})-x_i)= \sum_{i=1}^n (u_i+v_i-x_i).
\]

\begin{prop}
\label{push}
Let $(X(k); k \geq 0)$ be an $r$-interlacing random walk, starting from 
$X(0)$ having the distribution $q_s(y,dx)$ for some given $y \in A_n$,
where $q_s$ is defined by (\ref{q}).
Let the process $(Y(k); k \geq 0)$ be given by  $Y(0)=y$ and  
\[
Y(k+1)=\phi_{Y(k), X(k+1)} (X(k)), \text { for }  k \geq 0,
\]
where $\phi$ is defined by (\ref{phi}).
Then   $(Y(k); k \geq 0)$ is distributed as an $r$-interlacing random walk 
starting from $y$.
\end{prop}
\begin{proof}
We prove by induction on $m$ that the law of $\bigl(Y(1), \ldots Y(m), 
X(m)\bigr)$ is given by $q_r(y,dy(1))\ldots 
q_r(y(m-1),dy(m))q_s(y(m),dx(m))$. Suppose this holds for some $m$. Then, 
since $Y(1)\ldots Y(m)$ are measurable with respect to $X(0),X(1), \ldots 
X(m)$ the joint law of  $\bigl(Y(1), \ldots Y(m), X(m), X(m+1)\bigr)$ is 
given by
\[
q_r(y,dy(1))\ldots q_r(y(m-1),dy(m))q_s(y(m),dx(m))q_r(x(m),dx(m+1)).
\]
Equivalently we may say that the law of  $\bigl(Y(1), \ldots Y(m),  
X(m+1)\bigr)$ is 
\[
q_r(y,dy(1))\ldots q_r(y(m-1),dy(m))(q_s q_r)(y(m),dx(m+1))
\]
and that the conditional law of $X(m)$ given the same variables is uniform 
on $\tau_{Y(m) X(m+1)}$. From  the measure preserving properties of the 
maps $\phi_{u,v}$, it follows that, conditionally on  $\bigl(Y(1), \ldots 
Y(m), X(m+1)\bigr)$,    $Y(m+1)$ is distributed  uniformly on 
$\tau^\prime_{Y(m),X(m+1)}$, and the inductive hypothesis for $m+1$ 
follows from this and the commutation relation $q_sq_r=q_rq_s$. 
\end{proof}

The dynamics of the coupled processes $\bigl(X(k),Y(k); k \geq 0 \bigr)$ 
are illustrated in  the following two diagrams, in which interlacing 
configurations $x=X(k)$ and $y=Y(k)$ are shown together with updated 
configurations  $x^\prime=X(k+1)$ and $y^\prime=Y(k+1)$. It is natural to 
think of these as particle positions on ( a portion of ) the circle.
For simplicity we consider an example where $x^\prime$ and $x$ differ only 
in the $i$-th co-ordinate.

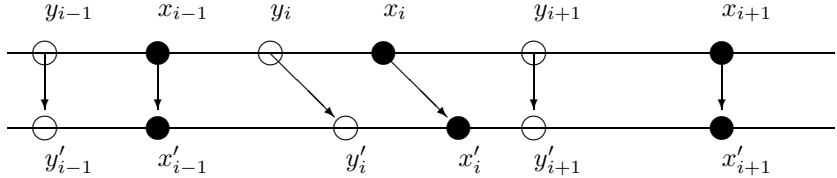
\setlength{\unitlength}{0.5cm}
\begin{picture}(14,6)
\put(1,2){\line(1,0){22}}
\put(1,4){\line(1,0){22}}
\put(2,1){\text{$y^\prime_{i-1}$}}
\put(2,5){\text{$y_{i-1}$}}
\put(2,2){\circle{0.6}}
\put(2,4){\circle{0.6}}
\put(2,4){\vector(0,-2){1.5}}
\put(5,1){\text{$x^\prime_{i-1}$}}
\put(5,5){\text{$x_{i-1}$}}
\put(5,2){\circle*{0.6}}
\put(5,4){\circle*{0.6}}
\put(5,4){\vector(0,-2){1.5}}
\put(10,1){\text{$y^\prime_{i}$}}
\put(8,5){\text{$y_{i}$}}
\put(10,2){\circle{0.6}}
\put(8,4){\circle{0.6}}
\put(8,4){\vector(1,-1){1.7}}
\put(13,1){\text{$x^\prime_{i}$}}
\put(11,5){\text{$x_{i}$}}
\put(13,2){\circle*{0.6}}
\put(11,4){\circle*{0.6}}
\put(11,4){\vector(1,-1){1.7}}
\put(15,1){\text{$y^\prime_{i+1}$}}
\put(15,5){\text{$y_{i+1}$}}
\put(15,2){\circle{0.6}}
\put(15,4){\circle{0.6}}
\put(15,4){\vector(0,-2){1.5}}
\put(20,1){\text{$x^\prime_{i+1}$}}
\put(20,5){\text{$x_{i+1}$}}
\put(20,2){\circle*{0.6}}
\put(20,4){\circle*{0.6}}
\put(20,4){\vector(0,-2){1.5}}
\end{picture}

The configuration $y^\prime$ is of course determined by $y$, $x$ and 
$x^\prime$ together. The simplest possiblity is shown above, in this case 
$i$th $y$ particle advances by the same amount as the $i$th $x$-particle.  
However should the $i$th $x$-particle advance beyond $y_{i+1}$ then it it 
pushes the $(i+1)$th $y$ particle along,
 whilst the increment in the position of the $i$th $y$ particle is limited 
to $y_{i+1}-x_i$.

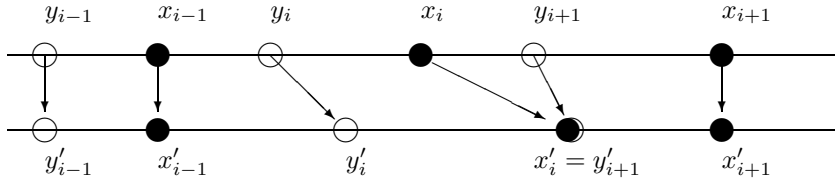
\begin{picture}(14,6)
\put(1,2){\line(1,0){22}}
\put(1,4){\line(1,0){22}}
\put(2,1){\text{$y^\prime_{i-1}$}}
\put(2,5){\text{$y_{i-1}$}}
\put(2,2){\circle{0.6}}
\put(2,4){\circle{0.6}}
\put(2,4){\vector(0,-2){1.5}}
\put(5,1){\text{$x^\prime_{i-1}$}}
\put(5,5){\text{$x_{i-1}$}}
\put(5,2){\circle*{0.6}}
\put(5,4){\circle*{0.6}}
\put(5,4){\vector(0,-2){1.5}}
\put(10,1){\text{$y^\prime_{i}$}}
\put(8,5){\text{$y_{i}$}}
\put(10,2){\circle{0.6}}
\put(8,4){\circle{0.6}}
\put(8,4){\vector(1,-1){1.7}}
\put(15,1){\text{$x^\prime_{i}=y^\prime_{i+1}$}}
\put(12,5){\text{$x_{i}$}}
\put(16.0,2){\circle{0.6}}
\put(15.9,2){\circle*{0.6}}
\put(12,4){\circle*{0.6}}
\put(12.3,3.8){\vector(2,-1){3}}
\put(15,5){\text{$y_{i+1}$}}
\put(15,4){\circle{0.6}}
\put(15,4){\vector(1,-2){0.8}}
\put(20,1){\text{$x^\prime_{i+1}$}}
\put(20,5){\text{$x_{i+1}$}}
\put(20,2){\circle*{0.6}}
\put(20,4){\circle*{0.6}}
\put(20,4){\vector(0,-2){1.5}}
\end{picture}

The  proof of Propostion \ref{push} made use of only  the measure 
preserving properties 
of the family of maps  $\phi_{uv}$, and consequently we can replace it by 
another family of maps having the same property and obtain a different 
coupling of the same processes.
The pushing interaction   in the coupling constructed above is also seen 
in the dynamics induced on Gelfand-Tsetlin patterns by the RSK 
correpondence. There exists a variant of the RSK algorithm (with column 
insertion replacing the more common row insertion), in which pushing is 
replaced by blocking. We will next describe a family 
$\psi_{uv}$ of measure-preserving maps that lead to a coupling with  such 
a  blocking interaction.

We recall first a version of the standard Skorohod lemma for periodic 
sequences. 
\begin{lem}[Skorohod] Suppose that $(z_i; i \in {\mathbf Z})$ is 
$n$-periodic and satisfies
$\sum_{i=1}^n z_i<0$. Then there exists a unique pair of $n$-periodic 
sequences $(r_i; i \in {\mathbf Z})$ and $(l_i; \in {\mathbf Z})$ such 
that
\[
r_{i+1}=r_i +z_{i}+l_{i+1} \qquad \text{ for all } i \in {\mathbf Z},
\] 
with the additional properties  $r_i \geq 0$, $l_i \geq 0$, and $ l_i>0 
\implies r_i=0$ for all $i \in {\mathbf Z}$.
\end{lem}

A configuration $ (x_1,x_2 \ldots x_n)$ of $n$ points on the circle  will 
be implicitly extended to a sequence $(x_i ; i \in {\mathbf Z})$ 
satisfying $x_{i+n}=x_i+2\pi$.

\begin{prop}
Suppose that $u$, $v$, and $x$ are three configurations on the circle with 
$u\preceq_s x \preceq_r v$, where $s>r$.
Define an $n$-periodic sequence via
\[
z_i=v_i-x_i-x_{i+1}+u_{i+1}  \qquad \text { for } i \in {\mathbf Z},
\]
and let $(r,l)$ be the  associated solution to the Skorohod problem.
Set $y_i= x_{i}-l_i$, then $(y_1, \ldots, y_n)$ is configuration on the 
circle  such that $u \preceq_r y \preceq_s v$. 
\end{prop}
\begin{proof}
Notice that $\sum_{i=1}^n z_i=r-s<0$ so the appeal to the Skorohod 
construction is legitimate.

Summing the Skorohod equation gives $ \sum_{i=1}^n z_i = - \sum_{i=1}^n 
l_i$. Consequently,  using the definition of $y$,
\[
\sum_{i=1}^n  ( y_i-u_i ) = \sum_{i=1}^n  (x_{i}-l_i-u_i)= s - 
\sum_{i=1}^n l_i= r.
\]
Similarly $\sum_{i=1}^n  ( v_i-y_i )=s$.

To conclude we will verify the inequalities
\[
\max(u_i,v_{i-1})\leq y_i \leq x_{i} \leq \min(u_{i+1}, v_i).
\]
It is easily checked that $z_{i-1}^-\leq \min(x_{i}-u_i, x_{i}-v_{i-1})$. 
The solution of the Skorohod problem satisfies
$0 \leq l_i \leq z_{i-1}^-$. Together with the definition of $y_i$, this 
gives the desired inequalities.
\end{proof}

By virtue of the preceeding result we may define $\psi_{u,v}: \tau_{uv} 
\rightarrow \tau^\prime_{uv}$ via $\psi_{uv}(x)= y$.

\begin{prop}
$\psi_{uv}$ is a measure-preserving bijection.
\end{prop}
\begin{proof}
 For $x \in A_n$, extended as before to a sequence $(x_i ; i \in{\mathbf 
Z}$, let $x^\dag$ be defined by $x^\dag_i=-x_{-i}$ for $ i \in {\mathbf 
Z}$. Observe that if $x \prec_r y$, then $y^\dag \prec_r x^\dag$. For 
$u,v,x$ and $y$ as in the previous proposition we will show that the 
application
\[
(u,x,v) \mapsto (v^\dag, y^\dag, u^\dag)
\]
is an involution, and this implies in particular that $\psi_{u,v}$ is 
invertible with  $x^\dag=\psi_{v^\dag u^\dag}(y^\dag)$. To this end first 
note that
$ v^\dag \prec_s y^\dag \prec_r u^\dag$, and hence it is meaningful to 
apply $\psi_{v^\dag u^\dag}$ to $y^\dag$. We must proof that the resulting 
conguration $\tilde{x}$ say, is equal to $x^\dag$. 

We have $ \tilde{x}_i = y^{\dag}_{i}-\tilde{l}_i$, where 
$(\tilde{r},\tilde{l})$ solves the Skorohod problem with data
\[
\tilde{z}_i=  u^\dag_i-y^\dag_i-y^\dag_{i+1}+v^\dag_{i+1}.
\]
Since $\tilde{x}_i=y^{\dag}_{i}-\tilde{l}_i= -x_{-i}+{l}_{-i}-\tilde{l}_i$  
verifying  that $\tilde{x}=x^\dag$ boils down to checking that 
$\tilde{l}_i=l_{-i}$. For this it is suffices, by the uniqueness property 
of solutions to the Skorohod problem, to confirm that   $r^\prime_i= 
r_{-i}$ and  $l^\prime_i=l_{-i}$  solve the Skorohod problem with data 
$\tilde{z}$. Now $r^\prime$ and $l^\prime$ are non-negative and satisfy 
$l^\prime_i>0 \implies r^\prime_i=0$ because $r$ and $l$ have these 
properties. We just have to compute 
\begin{align*}
r^\prime_{i}+ \tilde{z}_{i}+ 
l^\prime_{i+1}&=r_{-i}+u^\dag_i-y^\dag_i-y^\dag_{i+1}+v^\dag_{i+1}+l_{-(i+1)} 
\\
&=r_{-i}-u_{-i}+y_{-i}+y_{-(i+1)}-v_{-(i+1)} +l_{-(i+1)} \\
&=r_{-i}-u_{-i}+x_{-i}-l_{-i}+x_{-(i+1)}-l_{-(i+1)}-v_{-(i+1)}+l_{-(i+1)} 
\\
&=r_{-i}- z_{-(i+1)}-l_{-i}\\
&= r_{-(i+1)}=r^\prime_{i+1}.
\end{align*}
This proves the involution property.

It remains to verify the measure-preserving property. The construction of 
$\psi_{uv}$ is such that it is evident that it is a piecewise linear 
mapping, and that its Jacobian is almost everywhere integer valued. Since 
the same applies to the inverse map constructed from  $\psi_{v^\dag 
u^\dag}$ we conclude the 
Jacobian is almost everywhere $\pm 1$ valued.
\end{proof}

Finally the following propostion follows by exactly the same argument as 
Proposition \ref{push}.

\begin{prop}
\label{block} Suppose that $s>r$.
Let $(X(k); k \geq 0)$ be an $r$-interlacing random walk, starting from 
$X(0)$ having the distribution $q_s(y,dx)$  for some given $y \in A_n$.
Let the process $(Y(k); k \geq 0)$ be given by  $Y(0)=y$ and  
\[
Y(k+1)=\psi_{Y(k), X(k+1)} (X(k)). \text { for }  k \geq 0.
\]
Then   $(Y(k); k \geq 0)$ is distributed as an $r$-interlacing random walk 
starting from $y$.
\end{prop}

Let us illustrate this coupling in a similar fashion to before.

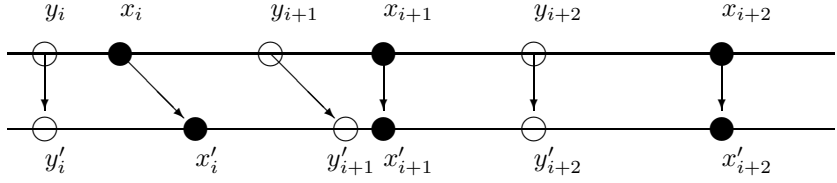
\setlength{\unitlength}{0.5cm}
\begin{picture}(14,6)
\put(1,2){\line(1,0){22}}
\put(1,4){\line(1,0){22}}
\put(2,1){\text{$y^\prime_{i}$}}
\put(2,5){\text{$y_{i}$}}
\put(2,2){\circle{0.6}}
\put(2,4){\circle{0.6}}
\put(2,4){\vector(0,-1){1.5}}
\put(6,1){\text{$x^\prime_{i}$}}
\put(4,5){\text{$x_{i}$}}
\put(6,2){\circle*{0.6}}
\put(4,4){\circle*{0.6}}
\put(4,4){\vector(1,-1){1.7}}
\put(9.5,1){\text{$y^\prime_{i+1}$}}
\put(8,5){\text{$y_{i+1}$}}
\put(10,2){\circle{0.6}}
\put(8,4){\circle{0.6}}
\put(8,4){\vector(1,-1){1.7}}
\put(11,1){\text{$x^\prime_{i+1}$}}
\put(11,5){\text{$x_{i+1}$}}
\put(11,2){\circle*{0.6}}
\put(11,4){\circle*{0.6}}
\put(11,4){\vector(0,-2){1.5}}
\put(15,1){\text{$y^\prime_{i+2}$}}
\put(15,5){\text{$y_{i+2}$}}
\put(15,2){\circle{0.6}}
\put(15,4){\circle{0.6}}
\put(15,4){\vector(0,-2){1.5}}
\put(20,1){\text{$x^\prime_{i+2}$}}
\put(20,5){\text{$x_{i+2}$}}
\put(20,2){\circle*{0.6}}
\put(20,4){\circle*{0.6}}
\put(20,4){\vector(0,-2){1.5}}
\end{picture}

In the  simplest case shown above, the  $(i+1)$th $y$ particle advances by 
the same amount as the $i$th $x$-particle.  However in the event that this 
would result in it passing the position $x_i$, then it is blocked at that 
point, and the unused part of the increment $(x^\prime_i-x_i)$ is passed 
to the $(i+2)$th $y$ particle as is shown beneath.

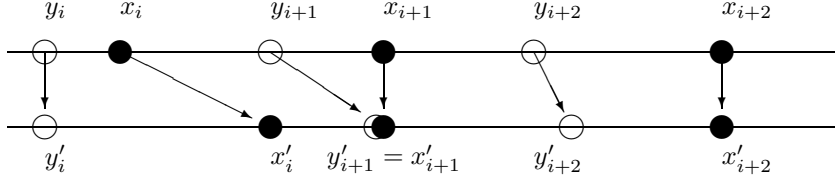
\begin{picture}(14,6)
\put(1,2){\line(1,0){22}}
\put(1,4){\line(1,0){22}}
\put(2,1){\text{$y^\prime_{i}$}}
\put(2,5){\text{$y_{i}$}}
\put(2,2){\circle{0.6}}
\put(2,4){\circle{0.6}}
\put(2,4){\vector(0,-2){1.5}}
\put(8,1){\text{$x^\prime_{i}$}}
\put(4,5){\text{$x_{i}$}}
\put(8,2){\circle*{0.6}}
\put(4,4){\circle*{0.6}}
\put(4,4){\vector(2,-1){3.5}}
\put(9.5,1){\text{$y^\prime_{i+1}=x^\prime_{i+1}$}}
\put(8,5){\text{$y_{i+1}$}}
\put(11,2){\circle*{0.6}}
\put(10.8,2){\circle{0.6}}
\put(8,4){\circle{0.6}}
\put(8,4){\vector(3,-2){2.4}}
\put(11,5){\text{$x_{i+1}$}}
\put(11,4){\circle*{0.6}}
\put(11,4){\vector(0,-2){1.5}}
\put(15,1){\text{$y^\prime_{i+2}$}}
\put(15,5){\text{$y_{i+2}$}}
\put(16,2){\circle{0.6}}
\put(15,4){\circle{0.6}}
\put(15,4){\vector(1,-2){0.8}}
\put(20,1){\text{$x^\prime_{i+2}$}}
\put(20,5){\text{$x_{i+2}$}}
\put(20,2){\circle*{0.6}}
\put(20,4){\circle*{0.6}}
\put(20,4){\vector(0,-2){1.5}}
\end{picture}

We can also consider discrete analogues of these constructions. Set 
$A_n^N= A_n\cap \Z^n/(2\pi N)$. Then, for $r\in \{1/(2\pi N), 2/(2\pi N), 
\ldots, (N-1)/(2\pi N)\}$ we may define  a Markovian transition  kernel 
$q^N_r$ on $A^N_n$ as follows.
Define 
\[
l^N_r(x,x^\prime)= \begin{cases} 1 & \text {if } z \preceq_r x^\prime, \\
0 &\text{otherwise.}
\end{cases}
\]
and let 
\begin{equation}
h^N(x)=\prod_{1 \leq l <m\leq n} |e^{i\bar{x}_l}-e^{i\bar{x}_m}|,
\end{equation}
where $\bar{x}_j= \frac{N}{N+n}(x_j+j/(2\pi N))$.
Then, as we shall see below, $l^N_r$, as a kernel on $A^N_n$, admits $h^N$ 
as a strictly positive  eigenfunction, and consequently  
\begin{equation}
q^N_r(x,x^\prime)=\frac{1}{\gamma^N_r} \frac{h^N(x^\prime)}{h^N(x)} l_r^N(x,x^\prime),
\end{equation}
defines a Markovian transition kernel, where $\gamma^N_r>0$ denotes the 
Perron-Frobenius eigenvalue of $l^N_r$.  We will call a Markov chain 
$(X(k); k \geq 0)$ with these transition probabilities, an  
$r$-interlacing random walk on on $A^N_n$.

In the special case $r=1/(2\pi N)$ the  $r$-interlacing random walk on on 
$A^N_n$ is closely related to non-coliding random walks on the circle, as 
considered in \cite{kor02} for instance. Specifically if $(X(k); k \geq 
0)$ is the walk on $A^N_n$ then the process given by
\[
\bigl( e^{i\bar{X}_1(k)}, e^{i\bar{X}_2(k)}, \ldots, 
e^{i\bar{X}_n(k)}\bigr),
\]
where $\bar{X}_j(k)= \frac{N}{N+n}(X_j(k)+j/(2\pi N))$,
gives a process of $n$ non-coliding random walks on the circle 
$\{e^{ik/(2\pi(N+n)}; k=0, 1, 2,\ldots ,N+n-1\}$.

Fix $r$ and $s \in \{1/(2\pi N), 2/(2\pi N), \ldots, (N-1)/(2\pi N)\}$.
For $u,v \in A^N_n$ we may consider the two sets $\tau^N_{u,v}=\tau_{u,v} 
\cap  \Z^n/(2\pi N)$ and $\tau^{N\prime}_{u,v}=\tau^\prime_{u,v} \cap  
\Z^n/(2\pi N)$, where $\tau_{u,v}$ and $\tau^\prime_{u,v}$ are the 
polygons   defined previously. The map $\phi_{u,v}$ carries $\Z^n/(2\pi 
N)$ into $\Z^n/(2\pi N)$, and  since we know it to be an isometry between 
$\tau_{u,v}$ and $\tau_{u,v}^\prime$ it must  therefore restrict to  a 
bijection between $\tau^N_{u,v}$  and $\tau^{N\prime}_{u,v}$. In 
particular the cardinalities of these two sets are equal.  Since $ 
l^N_sl^N_r(u,v)=|\tau^N_{u,v}|=|\tau_{u,v}^{N\prime}|=l^N_r l^N_s(u,v) $ 
we deduce that $l^N_r$ and $l^N_s$, and so $q^N_r$ and $q^N_s$ commute. 
This is useful in verifying our previous assertion that $h^N$ is a 
eigenfunction of $l^N_r$. The case $r=1$ may be easily deduced by 
direct calculation, see \cite{kor02} for a similar calculation. Moreover 
the statespace $A^N_n$ is irreducible for $l_r^N$ for $r=1$ (though not 
in general), and thus the Perron-Frobenius eigenfunction of $l^N_1$ is 
unique up to multiplication by scalars. Consequently it follows from the 
commutation relation that $h^N$ is an eigenfunction of $l^N_r$ for every 
$r$.

Using the fact that $\phi_{u,v}$ is a bijection between $\tau^N_{u,v}$  
and $\tau^{N\prime}_{u,v}$, we see that, substuting $\tau^N_{u,v}$ for 
$\tau_{u,v}$,  $q^N_r$ for $q_r$ and so on, the proof of Proposition 
\ref{push} applies verbatum in the case that $X$ is an $r$ interlacing 
random walk on $A^N_n$ rather than $A_n$. Similarly  $\psi_{u,v}$ is a 
bijection between $\tau^N_{u,v}$  and $\tau^{N\prime}_{u,v}$, and so 
Proposition \ref{block} holds also in the discrete case.

We can extend the above constructions, by Kolmogorov consistency,
to define a Markov chain $(\{X^{(m)}(k),\ m\in\Z\},\ k\in\Z)$ on the state space
$$\{ x\in A_n^\Z:\ x^{(m)} \preceq_r x^{(m+1)},\ m\in\Z\}$$
with the following properties:
\begin{enumerate}
\item For each $m\in\Z$, $(X^{(m)}(k),k\in\Z)$ is an $s$-interlacing random walk
\item For each $k\in\Z$, $(X^{(m)}(k),m\in\Z)$ is an $r$-interlacing random walk 
\item For each $m,k\in\Z$, 
$$X^{(m)}(k+1)=\phi_{X^{(m+1)}(k),X^{(m)}(k+1)}(X^{(m)}(k)).$$
\end{enumerate}
In the above, we can replace $\phi$ by $\psi$, and also consider the discrete versions.
It is interesting to remark on the non-colliding random walk case,
that is, the discrete model on $A^N_n$ with $r=s=1/2\pi N$.  In this case, the evolution
of the above Markov chain is completely deterministic.  Indeed, by ergodicity of the
non-colliding random walk, for each $k\in\Z$, with probability one, there exist infinitely 
many $m$ such that $X^{(m)}(k)\in \{ ((1+c)/2\pi N,\ldots,(n+c)/2\pi N),\ c\in\Z \}$ and for 
these $m$ there is only one allowable transition, namely to $X^{(m)}_j(k+1)=X^{(m)}_j(k)$
for $j\le n-1$ and $X^{(m)}_n(k+1)=X^{(m)}_n(k)+1/2\pi N$.  By property (3) above,
this determines $X^{(m)}(k+1)$ for all $m\in\Z$.

\section{Couplings of interlaced  Brownian motions}

In this section we describe  a  Brownian motion construction that can 
considered  as  a scaling limit
of  the coupled random walks  of the previous section.  

The Laplacian on $A_n$ with Dirichlet boundary conditions admits a 
unique non-negative eigenfunction, the function $h$ defined previously at 
\eqref{h} which corresponds to the greatest eigenvalue $\lambda_0=-n(n-1)(n+1)/12$. A 
$h$-Brownian motion on $A_n$, is a diffusion  with transition densities 
$Q_t$ given by
\begin{equation}
Q_t(x,x^\prime)= e^{-\lambda_0 t}\frac{h(x^\prime)}{h(x)}  
Q^0_t(x,x^\prime),
\end{equation}
where $Q^0_t$ are the transition densities of standard $n$-dimensional 
Brownian motion killed on exiting $A_n$, given explicitly by a continuous
analogue of the Gessel-Zeilberger formula, \cite{gz,hw}. 
If $(X(k); k \geq 0)$ is an $h$-Brownian motion in $A_n$ then
$\bigl( e^{iX_1(t)}, e^{iX_2(t)}, \ldots, 
e^{iX_n(t)}\bigr)$
gives a process of $n$ non-coliding Brownian motions on the circle, the 
same as arises as the eigenvalue process of Brownian motion in the unitary 
group, see \cite{dys62b}.

We wish to construct a bivariate process $(X,Y)$ with each of $X$ and $Y$ 
distributed as $h$-Brownian motions in $A_n$, and  with $Y(t) \preceq_s 
X(t)$, for some fixed $s \in (0,2\pi)$.
The dynamics we have in mind is based on the map $\phi$ of the previous 
section. $Y$ should be  deterministically constructed from its starting 
point and the trajectory of $X$, with $Y_i(t)$ tracking $X_i(t)$  except 
when this would cause the   interlacing constraint to be broken. In the 
following we set $\theta_{0}(t)= \theta_n(t)$.
\begin{prop}
\label{yfromx}
Given a continuous path $\bigl(x(t); t \geq 0 \bigr)$ in $A_n$, and a 
point $y(0)\in A_n$ such that $y(0) \preceq_s x(0)$, there exist unique 
continuous paths $\bigl(y(t); t \geq 0 \bigr)$ in $A_n$ and 
$\bigl(\theta(t); t \geq 0 \bigr)$ in ${\mathbf R}^n$ such that
\begin{itemize}
\item[(i)] $y_i(t) \preceq_s x_i(t)$  for all $t\geq 0$;
\item[(ii)]  $y_i(t)-y_i(0)= x_i(t)- x_i(0)+ 
\theta_{i-1}(t)-\theta_{i}(t)$;
\item[(iii)] for $i=1,2,\ldots n $, the real-valued process $\theta_i(t)$ 
starts from $\theta_i(0)=0$, is increasing, and the measure $d\theta_i(t)$ 
is supported on the set $\{ t: y_{i+1}(t)= x_{i}(t)\}$. 
\end{itemize}
\end{prop}

Our principal tool in proving this proposition is another variant of the  
Skorohod reflection lemma. For $l>0$, let $A_n(l)=\bigl\{ x \in {\mathbf 
R}^n: x_1\leq x_2 \leq  \ldots  \leq x_n\leq x_1+ l\bigr\}
$. We adapt our usual  convention,  letting  $x_{n+1}=x_1+l$ and 
$x_0=x_n-l$. 
\begin{lem}
\label{skorohod2}
Given a continuous path $\bigl(u(t); t \geq 0 \bigr)$ in ${\mathbf 
R}^n$ starting from $u(0)=0$, and a point $v(0) \in {A}_n(l)$ then there 
exists a unique pair of ${\mathbf R}^n$-valued, continuous  paths $(v, 
\theta)$  starting from $(v(0),0)$ with
\begin{itemize}
\item[(i)]
$v(t) \in  {A}_n(l)$  for every $t \geq 0$,
\item[(ii)] each component $\theta_i(t)$  increasing with the measure 
$d\theta_i(t)$ supported on the set $\{ t: v_i(t)=v_{i+1}(t)\}$,
\item[(iii)]  $v_i(t)=v_i(0)+ u_i(t)- \theta_i(t).$
\end{itemize}
\end{lem}

\begin{proof} 
Fix $T>0$.
The path $u$ is uniformly continuous on $[0,T]$, and so there exists 
$\epsilon>0$ such that $\sum_i |u_i(t)-u_i(s)| < l/n^2$ for all $|s-t| < 
\epsilon$. We  prove existence and uniqueness over the time interval 
$[0,\epsilon]$, and then repeat the argument over consecutive time 
intervals $[\epsilon,2\epsilon]$ etc. 

Observe that there must exist $k \in \{1,2,\ldots, n\}$ such that 
$v_{k+1}(0)-v_{k}(0) \geq l/n$. Set $v_k(t)= v_k(0)+u_k(t)$ and 
$\theta_k(t)=0$ for $t \in [0,\epsilon]$. Then we the usual Skorohod lemma 
applied  sucessively to give $v_{k-1}, v_{k-2}, \ldots, v_1,v_n, \ldots 
v_{k+1}$, in particular we let
\[
\theta_i(t)=\sup_{s \leq t} \bigl( v_{i+1}(t)- u_i(t)-v_i(0)\bigr)^-.
\]
The choice of $\epsilon$ is such that the $v_{k-1}$ so constructed 
satisfies $v_{k}(t) < v_{k+1}(t)$ for all $t \in [0,\epsilon]$, and 
consequently all the desired properties of $v$ and $ \theta$ hold.  For 
uniqueness, we note that any $(v,\theta)$ for which the desired properties 
holds must be equal to the one just constructed, which  follows from the 
uniqueness for the usual Skorohod construction.
\end{proof}

\begin{proof}[Proof of Proposition \ref{yfromx}]
This  is based on applying the Skorohod mapping to the path $u$ specified 
from $x$ by $u_i(t)= x_i(t)-x_i(0)$.

Choose $v(0) \in {A}_n(2\pi-s)$ so that $v_i(0)- v_{i-1}(0)= 
y_i(0)-x_{i-1}(0)$, and let $\bigl(v, \theta\bigr)$ be  given by the 
Skorohod mapping for the domain $A_n(2\pi-s)$ with data $u$ and $v(0)$.  
Then  for $t \geq 0$ define
 \begin{align*}
 \label{yfromx}
 y_i(t)=y_i(0)+u_i(t)+ \theta_{i-1}(t)- \theta_{i}(t). 
\end{align*} 
We claim that 
\begin{itemize}
\item[(i)] $y(t) \in A_n $ with $y(t)\preceq_s x(t)$ for every $t\geq 0$,
\item[(ii)] the measure $d\theta_i(t)$ is supported on the set $\{ t: 
y_{i+1}(t)=x_{i}(t)\}$.
\end{itemize}
Calculate as follows.
\begin{multline*}
y_{i}(t)- x_{i-1}(t)=\bigl(y_i(0)+ u_i(t)+ \theta_{i-1}(t)- 
\theta_{i}(t)\bigr)-\bigl(x_{i-1}(0)+u_{i-1}(t)\bigr)\\
=\bigl( v_{i}(0)+u_{i}(t)-\theta_{i}(t)\bigr)- 
\bigl(v_{i-1}(0)+u_{i-1}(t)-\theta_{i-1}(t)\bigr)  
= v_{i}(t)-v_{i-1}(t).
\end{multline*}
This is valid even if $i=0$ when we adhere to our conventions regarding 
$v_{0}$ etc.
Assertion (ii) above  follows since we see that $\{ t: 
y_i(t)=x_{i-1}(t)\}=\{t:v_i(t)=v_{i-1}(t)\}$, and this latter set carries 
$d\theta_{i-1}$.
Also since  $v_{i}(t)\geq v_{i-1}(t)$ we obtain one part of the 
interlacing condition, namely,
$y_i(t)\geq x_{i-1}(t)$. For the  other part, consider  the equality
\[
x_{i}(t)-y_i(t)= x_i(0)-y_i(0) + \theta_i(t)- \theta_{i-1}(t).
\]
Since this quantity is initially $x_i(0)-y_i(0)> 0$ and decreases only 
when 
$\theta_{i-1}$ increases it follows that if there exists an instant $t_1$ 
for 
which $x_i(t_1)-y_i(t_1)<0$, then there exists another instant, $t_0$, for 
which $x_i(t_0)-y_i(t_0)<0$  and which  belongs to the support of 
$d\theta_{i-1}$. The latter implies that $y_{i}(t_0)=x_{i-1}(t_0)$, and 
thus  
$x_{i}(t_0)<y_i(t_0)=x_{i-1}(t_0)$ which contradicts $x(t_0) \in A_n$. 
This proves existence. Uniqueness follows from  the uniqueness statement 
in 
Proposition \ref{skorohod2}.
\end{proof}

By virtue of this result we may make the following definition. For $y(0) 
\in A_n$,  let $\Gamma_{y(0)}$ be the application which applied to an 
$A_n$-valued path $\bigl(x(t) \geq 0\bigr)$ returns the path $\bigl(y(t); 
t \geq 0 \bigr)$ specified by Propostion \ref{yfromx}. The main result of 
this section is the following.
\begin{prop}
\label{bm}
Let $(X(t); t \geq 0)$ be an $h$-Brownian motion in $A_n$, starting from a 
point
$X(0)$ having the distribution $q_s(y,dx)$  for some given $y=y(0) \in 
A_n$.
Then the process   $Y=\Gamma_{y(0)}( X)$ is distributed as a $h$-Brownian 
motion in $A_n$.
\end{prop}

The domain $A_n$ is unbounded, which is a nuisance when we come to the 
probability, where life is easier if we have  a finite invariant  measure 
on the state space. Everything we do is invariant to shifts along the 
diagonal of ${\mathbf R}^n$, and  it is useful  to  project onto the 
hyperplane $H_0=\bigl\{ x \in {\mathbf R}^n: \sum x_i=0\bigr\}$.   For $x 
\in {\mathbf R}^n$,   the orthogonal projection of $x$ onto $H_0$ is given 
by
$ x \mapsto x- \bar x {\mathbf 1}$ where ${\mathbf 1} \in {\mathbf R}^n$ 
is the vector with every component equal to $1$, and $\bar{x}= 
n^{-1}\sum_i x_i$.
If $\bigl(X(t); t \geq 0\bigr)$ is an $h$-Brownian motion in $A_n$, then 
its projection onto $H_0$ is itself a diffusion process and indeed can be 
described as the $h$-transform of an  $(n-1)$-dimensional   Brownian 
motion in  $H_0$, killed on exiting $A_n \cap H_0$.  Introducing more 
generally   $H_s=\bigl\{ x \in {\mathbf R}^n: \sum x_i=s\bigr\}$, and the 
notion of an $h$-Brownian motion on $A_n \cap H_s$, we have the following 
variant of Propostion \ref{bm}.

\begin{prop}
\label{bm2}
Let $(X(t); t \geq 0)$ be an $h$-Brownian motion in $A_n \cap H_{s/2}$, 
starting from a point
$X(0)$ having the distribution $q_s(y,dx)$  for some given $y=y(0) \in A_n 
\cap H_{-s/2}$.
Then the process   $Y=\Gamma_{y(0)}( X)$ is distributed as a $h$-Brownian 
motion in $A_n\cap H_{-s/2}$.
\end{prop}

 Proposition \ref{bm} is easily deduced from this variant  using 
the fact that if $\bigl(X(t); t \geq 0\bigr)$ is an $h$-Brownian motion in 
$A_n$, then 
the projection of $X$ onto $H_0$ is independent of the process 
$n^{-1/2}\sum_i X_i(t)$, which is a one-dimensional Brownian motion.

The  results in the previous section were proved using the
measure-preserving properties associated with the dynamics for a single 
update. Since here we are working with continuous time processes such one
time-step methods are not applicable.
The  use of the Skorohod lemma in the construction of $\Gamma$ suggests
 that there may be a role to be played by a certain reflected Brownian 
motion.  We describe next how, adapting the
idea of Proposition \ref{yfromx} slightly, we can construct  interlaced 
processes $X$ and $Y$ from a reflected Brownian motion
$R$ in the domain $H_0\cap {A}_n(2\pi-s)$. Then it will turn out that time 
reversal properties of $R$ can be used to prove
Proposition \ref{bm2}.

Let $E(s)= \{ (x,y) \in (A_n \cap H_{s/2})  \times (A_n \cap H_{-s/2}):  y 
\preceq_s x\} $. We can introduce new co-ordinates on $E(s)$ as follows. 
For $(x,y)\in E(s)$ we let  $f (x,y)$ be  the unique $(r,l) \in H_0 \times 
H_0  $ such that
\begin{align}
r_{i+1}-r_i&=y_{i+1}-x_i, \\
l_{i+1}-l_{i}&=x_{i+1}-y_{i+1},
\end{align}
where $r_{n+1}= r_1+(2\pi -s)$ and $l_{n+1}= l_1+s$.
It is easily seen that $f:E(s) \rightarrow (H_0\cap {A}_n(2\pi-s)) \times 
( H_0\cap {A}_n(s))$ is bijective.

Now we construct  a process in $E(s)$  via these alternative co-ordinates.  
Begin by letting $\bigl( U(t); t \geq 0 \bigr)$ be a standard Brownian 
motion in 
${\mathbf R}^n$ starting from zero, and let $V(0)$ be an independent 
random variable, uniformly distributed in $H_0\cap {A}_n(2\pi-s)$. Let $ 
\bigl(V, \Theta \bigr)$ be determined from $U$ and $V(0)$ by applying the 
Skorohod mapping  for $A_n(2\pi-s)$ as given  in
Proposition \ref{skorohod2}.  Finally let $R(t)$  be the projection of 
$V(t)$ onto $H_0$, so $R(t)=V(t)-\bar{V}(t){\mathbf 1}$.
The process
$\bigl(R(t); t\geq 0\bigr)$ is, by construction, a semimartingale 
reflecting 
Brownian motion in  the polyhedron ${A}_n(2\pi-s)\cap H_0$. See 
\cite{rwilliams} for the general theory of such processes. 
Next we introduce a process $L$ also taking value in $H_0$, which is 
constructed out  a random initial value $L(0)$ together with  the 
increasing processes $\Theta_i$ for $i=1,2, \ldots,n$. Choose $L(0)$
independent of $R$ and uniformly distributed on $A_n(s) \cap H_0$, and let 
$L(t)$ be given by
\[
L(t)=  L_i(0)- \pi\Theta(t),
\]
where $\pi:{\mathbf R}^n \rightarrow H_0$ denotes the  projection onto $H_0$ defined by $\pi x= x -
\bar{x} {\mathbf 1}$.
Define the stopping time $\tau= \inf\{ t \geq 0: L(t) \not\in A_n(s)\}$. 
Then for $t \leq \tau$, we define $X(t)$ and $Y(t)$ by $\bigl(R(t), L(t) 
\bigr)= f\bigl(X(t), Y(t) \bigr)$.  The joint law of $ \bigl(X(t \wedge 
\tau ), Y(t \wedge \tau);  t \geq 0 \bigr)$ may be described  as follows.
\begin{enumerate}
\item $(X(0), Y(0))$ is uniformly distributed on $E(s)$.
\item $ \bigl(X(t \wedge \tau); t \geq 0)$ is distributed as a Brownian 
motion in $H_0$ stopped at the instance it first leaves $A_n$, and 
conditionally independent of $Y(0)$ given $X(0)$.
\item  $Y= \Gamma_{Y(0)} X$.
\end{enumerate}

We now turn to the time reversibility  of $R$. 
For  any vector $x \in {\mathbf R}^n$ we  will denote by $x^\dag$ the vector given by 
$x^\dag_i=-x_{n-i+1} $. Note that if $x \in H_0 \cap A_n(2\pi-s) $ then $x^\dag \in H_0 \cap A_n(2\pi-s)$ also. We also define, for any $x \in {\mathbf R}^n$, the vector
$x^\ddag$ via 
$x^\ddag_i=-x_{n-i}-s/n $ for $1=1,2, \ldots (n-1)$ and $x^\ddag_n=-x_n+s(n-1)/n$. Note that if $x \in H_{0} \cap A_n(s) $ then $x^\ddag \in H_{0} \cap A_n(s)$ also.

\begin{prop} Fix some constant  $T>0$. 
Let the processes $R$ and $L$ be  as above and let $\Lambda$ be the event 
$\{ L(t)\in A_n(s) \text{ for all } t \in 
[0,T]\}$. Then 
conditionally on $\Lambda$,
\[
\bigl( R(t), L(t); t \in [0,T] \bigr) \stackrel{law}{=} \bigl( 
R^\dag(T-t), L^\ddag(T-t) ; t \in [0,T] \bigr).
\]
\end{prop}

\begin{proof}

As remarked above $R$ is a semimartingale reflected Brownian motion in the 
polyhedral domain $A_n(2\pi -s) \cap H_0$. Indeed it satisfies
\[
R(t)= R(0)+B(t)+ \frac{1}{\sqrt{2}}  \sum \Theta_i(t) v^i,
\]
where $B$ is a standard Brownian motion in $H_0$ and the vector $v^i$ 
describes the direction of reflection 
associated with the face $F_i=\{ x \in  A_n(2\pi -s) \cap H_0: 
x_i=x_{i+1}\}$. Let $n^i$ be the inward facing unit normal to this face. 
Then an easy calculation shows that $v^i$, which is normalized so that the 
inner product $n^i \cdot v^i=1$, is given by $v^i=n^i+q^i$ where the $j$th 
component of the vector $q^i$  is given by
\[
q^i_j= \begin{cases}
\sqrt{2}/n-1/\sqrt{2} & \text{ if } j=i,i+1,\\
\sqrt{2}/n &\text{ otherwise.}
\end{cases}
\]
We observe that the skew-symmetry condition,
\[
n^i \cdot q^j+ q^i \cdot n^j =0,
\]
for all $i \neq j $, is met. 
Consequently by Theorem 1.2 of \cite{rwilliams}, the reflected Brownian 
motion $R$ is in duality relative to Lebesgue measure to another reflected 
Brownian motion on $H_0\cap {A}_n(2\pi-s)$ with  direction of reflection 
from the face  $F_i$ being $n^i-q^i$. It is not difficult to check that
$R^\dag$ is such a reflected Brownian motion.  Thus
\begin{equation*}
\bigl( R(t); t \in [0,T] \bigr) \stackrel{law}{=} \bigl( R^\dag(T-t) ; t 
\in [0,T] \bigr).
\end{equation*}
The process $2\Theta_i(t)$ is
the local time of $R_{i+1}(t)-R_i(t)$ at zero, and can be represented as  
\begin{equation*}
\lim_{ \epsilon \downarrow 0} \frac{1}{\epsilon} \int_0^t {\mathbf 
1}\bigl( R_{i+1}(s)- R_i(s) \leq \epsilon\bigr)ds.
\end{equation*}
Now  note that,
\begin{multline*}
\int_0^t {\mathbf 1}\bigl( R_{i+1}(s)- R_i(s) \leq \epsilon\bigr)ds
\stackrel{law}{=} \int_0^t {\mathbf 1}\bigl( R^\dag_{i+1}(T-s)- 
R^\dag_i(T-s) \leq \epsilon\bigr)ds \\
= \int_{T-t}^T {\mathbf 1}\bigl( R_{n-i+1}(s)- R_{n-i}(s) \leq 
\epsilon\bigr)ds
\end{multline*}
From this 
we deduce that the time reversal property extends to
\begin{equation*}
\bigl( R(t), \Theta(t); t \in [0,T] \bigr) \stackrel{law}{=} \bigl( 
R^\dag(T-t), \Theta^\sharp(T)- \Theta^\sharp(T-t); t \in [0,T] \bigr),
\end{equation*}
where $ \Theta^\sharp_i(t)= \Theta_{n-i} $ for $i=1,2,\ldots,(n-1)$ and 
$\Theta^\sharp_n=\Theta_n$.

Let $F$ be a bounded path functional.  Let $v_n(s)$ denote the Lebesgue measure of $H_0 \cap A_n(s)$. Then using the time reversal property,
\begin{multline*}
{\mathbf E} \left[ F\bigl( R(t), L(t); t \in [0,T] \bigr){\mathbf 
1}_\Lambda \right] = \\
\frac{1}{v_n(s)}{\mathbf E} \left[ \int_{H_0} d\alpha F\bigl( R(t), 
\alpha-\pi \Theta(t); t \in [0,T] \bigr)
{\mathbf 1}{\{ \alpha-\pi\Theta(t) \in A_n(s) ; t \in[0,T]\}} \right] =
\\
\frac{1}{v_n(s)}{\mathbf E} \Biggl[ \int_{H_0} d\alpha F\bigl( 
R^\dag(T-t), \alpha- \pi\Theta^\sharp(T)+\pi\Theta^\sharp(T-t); t \in 
[0,T] \bigr) \Biggr. 
\\ \Biggl. \times
{\mathbf 1}{\bigl\{ \alpha-\pi\Theta^\sharp(T)+\pi\Theta^\sharp(T-t) \in 
A_n(s) ; t \in[0,T] \bigr\}} \Biggr].
\end{multline*}

Now we  make the substitution $\hat{\alpha} = \alpha-\pi\Theta^\sharp(T)$ 
to obtain
\begin{multline*}
\frac{1}{v_n(s)}
{\mathbf E} \Biggl[ \int_{H_0} d\hat{\alpha} F\bigl( R^\dag(T-t), 
\hat{\alpha}+ 
\pi\Theta^\sharp(T-t) ; t \in [0,T] \bigr) 
\Biggr. \\  \Biggl. \times{\mathbf 1}{\{ 
\hat{\alpha}+\pi\Theta^\sharp(T-t) 
\in A_n(s) , t \in[0,T]\}} \Biggr]=\\
{\mathbf E} \left[ F\bigl( R^\dag(T-t), L^\ddag(T-t); t \in [0,T]
\bigr){\mathbf 1}_\Lambda \right],
\end{multline*}
where we have used $L^\ddag(t)=L^\ddag(0)+ \pi \Theta^\sharp(t)$.
\end{proof}

\begin{proof}[Proof of Proposition \ref{bm2}]
Let $R$, $\Theta$ and $L$ be as above, and once again let $\Lambda$ be the 
event that $\big\{L(t) \in A_n(s) \text { for all } t \in [0,T]\bigr\}$. 
Recall  the  mapping $f$ such that
$\bigl(R(t), L(t) \bigr)= f\bigl(X(t), Y(t) \bigr)$. It is easily verified 
that
$ \bigl(Y^\dag(t),X^\dag(t)\bigr) \in E(s)$ and that
\[
\bigl(R^\dag(t), L^\ddag(t) \bigr)= f\bigl(Y^\dag(t), X^\dag(t) \bigr).
\]
Thus the preceeding time reversal result implies that, conditionally on 
$\Lambda$,
\[
\bigl( X(t), Y(t); t \in [0,T] \bigr) \stackrel{law}{=} \bigl( 
Y^\dag(T-t),X^\dag(T-t) ; t \in [0,T] \bigr).
\]
For the final step of the argument we consider $X$ and $Y$ be as 
above and denote the governing measure by ${\mathbf P}$. Then we let 
\[
\tilde{\mathbf P}= \frac{e^{-\lambda_0 T}}{\gamma_r}{\mathbf 1}_{\Lambda} h(Y(0))h(X(T)) \cdot 
{\mathbf P}.
\]
 Under $\tilde{\mathbf P}$,  the 
equality in law,
\[
\bigl( X(t), Y(t); t \in [0,T] \bigr) \stackrel{law}{=} \bigl( 
Y^\dag(T-t),X^\dag(T-t) ; t \in [0,T] \bigr).
\] holds unconditionally. 
Finally we note that under  $\tilde{\mathbf P}$, the distribution of 
$\bigl( X(t), t \in [0,T] \bigr)$ and hence  of $\bigl( Y^\dag(T-t), t \in 
[0,T] \bigr)$ is that of a stationary $h$-Brownian motion on $A_n \cap 
H_{s/2}$  
But this latter law is invariant under time reversal, and  its image under 
the conjugation $x \mapsto x^\dag$ is the law of  a stationary 
$h$-Brownian motion on $A_n \cap H_{-s/2}$. This is therefore   the law of 
$\bigl( Y(t), t \in [0,T] \bigr)$.  Conditioning on $Y(0)$ gives the 
statement of the proposition.
\end{proof}

In the case $n=2$, the results of this section can be expressed in terms of
Brownian motion in a compact interval. Let $X=(X_t,t\ge 0)$
be a Brownian motion conditioned, in the sense of Doob, never to exit the interval
$[-p,p]$, where $p>0$.  Let $y\in [-p,p]$, $a\in [0,2p]$ and suppose that the initial law
of $X$ is supported in the interval $[|y+a-p|-p,p-|y+p-a|]$ with density proportional
to $\cos(\pi x/2p)$. Let $Z$ be the image of the path $X+(y-X_0+a)/2$
under the Skorohod reflection map for the interval $[0,a]$. In other words,
$$Z_t=X_t+(y-X_0+a)/2+L_t-U_t,$$
where $L$ and $U$ are the unique continuous, non-decreasing paths such that the points
of increase of $L$ occur only at times when $Z_t=0$, the points of increase of $U$
occur only at times when $Z_t=a$, and $Z_t\in [0,a]$ for all $t\ge 0$.
Then the process
$$Y_t=y-X_0+X_t+2(L_t-U_t)\qquad t\ge 0$$ is a Brownian motion conditioned, in the
sense of Doob, never to exit the interval $[-p,p]$.  This is a special case of
Proposition~\ref{bm2}.  Actually, in the statement of that proposition we have
$p=\pi$, but this can be easily modified for general $p$.  It is interesting to consider
this statement when $y=0$ and $p\to\infty$. Then $X$ is a standard Brownian
motion, initially uniformly distributed on the interval $[-a,a]$.  The process $Z$
is a reflected Brownian motion in $[0,a]$, initially uniformly distributed on $[0,a]$.
The conclusion in this case is that $Y$ is a standard Brownian motion started from
zero.  We remark that in this setting, if instead we take $X_0=-a$, then $Y$ is a
Brownian motion started from zero, conditioned (in an appropriate sense) to hit $a$
before returning to zero.  This is a straightforward consequence of the above result
(for uniform initial law) and the fact (see~\cite{rp81}) that, if we set $T=\inf\{t\ge 0:\ Y_t=a\}$,
then the law of $X_T$ is uniform on $[-a,a]$.  Note that if we let $a\to\infty$ in this
case we recover Pitman's representation for the three-dimensional Bessel process.
There are explicit formulae for the Skorohod reflection map for
the compact interval $[0,a]$ and hence for the process $Y$ in the above discussion.
Let $f(t)=X_t+(y-X_0+a)/2$ and write $f(s,t)=f(t)-f(s)$.  A discrete version of the
Skorohod problem was considered in~\cite{toomey}, from which we deduce the
expressions
\begin{eqnarray*}
Z_t&=&\max\left\{ \sup_{0\le r\le t}  \min\{ f(r,t),a+\inf_{r\le s\le t} f(s,t)\} ,
\ \min\{ f(t),a+\inf_{0<s<t}f(s,t)\}\right\} \\
&=& \min\left\{ \inf_{0\le r\le t}  \max\{ a+f(r,t),\sup_{r\le s\le t} f(s,t)\} ,
\ \max\{ f(t),\sup_{0<s<t}f(s,t)\}\right\} .
\end{eqnarray*}
An alternative formula was obtained in~\cite{klrs}, which yields
\[ Z_t=\phi(t)-\sup_{0\le s\le t}\biggl[\bigl(\phi(s)-a\bigr)^+\wedge\inf_{s\le u\le t}\phi(u)\biggr]\]
where \[\phi(t)=f(t)+\sup_{0\le s\le t}[-f(s)]^+.\]
It could be interesting to relate the corresponding expressions for $Y_t$
to the Pitman transforms introduced in \cite{bbo05}.

\section{A bead model on the cylinder} 

In this section it will be convenient to work with a slightly weaker notion of
interlacing, defined as follows.
For $a,b\in D_n$, write $a\prec b$ if 
$$a_1\le b_1< a_2 \le \cdots < a_n \le b_n,$$  and $a \succ b$ if
$$b_1< a_1 \le b_2 < \cdots \le b_n < a_n.$$
For $y=\{e^{ia_1},\ldots,e^{ia_n}\}$ and $z=\{e^{ib_1},\ldots,e^{ib_n}\}$, 
where $a,b\in D_n$, write $y\prec z$ if either $a\prec b$ or $a\succ b$,
and define
$$l(y,z)=\begin{cases} \sum_j (b_j-a_j) & \mbox{if } \sum_j (b_j-a_j)\ge 0 \\
\sum_j (b_j-a_j)+2\pi & \mbox{otherwise.}\end{cases}$$

Consider the Markov kernels defined, for $q>0$, by
\begin{equation}\label{mdef}
m_q(y,dz) = c_q^{-1} \int_0^{2\pi} |1-e^{ir}|^{n-1} q^r p_r(y,dz) dr ,
\end{equation}
where $p_r$ is defined by (\ref{pr-def}) and
$$c_q=\int_0^{2\pi} |1-e^{ir}|^{n-1} q^r dr.$$
By Proposition \ref{br}, if we define
$$\I_q(y,z)=\begin{cases} q^{l(y,z)} & \mbox{ if 
$y\prec z$,}\\ 0 & \mbox{ otherwise,}\end{cases}$$
then
\begin{equation}\label{mi}
m_q(y,dz) = \tilde{c}_q^{-1} \frac{\Delta(z)}{\Delta(y)} \I_q(y,z) dz,
\end{equation}
where $\tilde c_q=c_q/(n-1)! $.
Recall that $\mu(dx)=(2\pi)^{-n}\Delta(x)^2dx$ is the probability measure on $C_n$ 
induced from Haar measure on $U(n)$.
The Markov chain with transition density $m_q$ has $\mu$ as an invariant
measure and, with respect to $\mu$, has time-reversed transition probabilities
$$ \overline{m}_q(z,dy)= \tilde c_q^{-1} \frac{\Delta(y)}{\Delta(z)} \I_q(y,z) dy .$$
We can thus construct a two-sided
stationary version of this Markov chain to obtain a probability measure $\a$
on $C_n^\Z$, supported on configurations $\cdots\prec x^{-1}\prec x^0\prec x^1\prec x^2
\prec \cdots$.  We will show that $\a$ defines
a determinantal point process on $[0,2\pi)^\Z$.
By stationarity it suffices to consider the restrictions 
$\a_m$ to the cylinder sets $C_{n,m}:=C_n^{\{1,2,\ldots,m\}}$. 
Writing $\bar{x}=(x^1,\ldots,x^m)$ and $d\bar{x}=dx^1\cdots dx^m$,
\begin{equation}\label{alpha}
\a_m(d\bar{x}) = \mu(dx^1)m_q(x^1,dx^2)\cdots m_q(x^{m-1},dx^m) .
\end{equation}
Assume for the moment that $q\ne 1$.  Define a function $f:\R\to\C$ by
$$f(u)=\frac{(qe^{i(n-1)/2})^{u\mod 2\pi}}{1-(-1)^{n-1}q^{2\pi}}.$$
\begin{lem}\label{iform}
For $y=\{e^{ia_1},\ldots,e^{ia_n}\}$ and $z=\{e^{ib_1},\ldots,e^{ib_n}\}$, 
where $a,b\in D_n$,
$$\I_q(y,z)= (1-(-1)^{n-1}q^{2\pi}) e^{ i \frac{n-1}{2}\sum_j(a_j-b_j)} 
\det\left( f(b_k-a_j) \right)_{1\le j,k\le n} .
$$
\end{lem}
\begin{proof}
Let $c\ne 1$ and consider the $n\times n$ matrix $W=(w_{jk})$ defined by
$$w_{jk}=\begin{cases} 1 & a_j \le b_k \\
c & a_j>b_k.\end{cases}$$
If $a\prec b$, $W$ consists of 1's on and above the diagonal
and $c$'s below, so that $\det W=(1-c)^{n-1}$.
If $a\succ b$, $W$ consists of 1's above the diagonal
and $c$'s on and below the diagonal, so that $\det W=c(c-1)^{n-1}$.
If neither $a\prec b$ or $a\succ b$, then there must exist an index $j$ such that, either
$a_j=a_{j+1}$, or $a_j<a_{j+1}\le b_k$ for all $k$, or $b_k<a_j<a_{j+1}\le b_{k+1}$
for some $k$.  In each of these cases, rows $j$ and $j+1$ of $W$ are identical
and hence $\det W=0$. Thus,
\begin{equation}\label{int}
\det W = \begin{cases} (1-c)^{n-1} & a\prec b\\
c(c-1)^{n-1} & a\succ b \\
0 & \mbox{otherwise.}\end{cases}
\end{equation}
Taking $c=(qe^{i(n-1)/2})^{2\pi}=(-1)^{n-1}q^{2\pi}$, we can write
\begin{eqnarray*}
(1-c) e^{ i \frac{n-1}{2}\sum_j(a_j-b_j)} \det\left( f(b_j-a_k) \right)_{1\le j,k\le n} &=&
q^{\sum_j(b_j-a_j)} (1-c)^{-(n-1)} \det W \\
&=& \begin{cases} 
q^{\sum_j(b_j-a_j)} & a\prec b \\
q^{\sum_j(b_j-a_j)+2\pi} & a\succ b \\
0 & \mbox{otherwise}\end{cases} \\
&=& \I_q(y,z) ,
\end{eqnarray*}
as required.
\end{proof}
For $r=1,\ldots,m-1$, define $\phi_{r,r+1}:[0,2\pi)^2 \to\C$ by
$\phi_{r,r+1}(a,b)=f(b-a)$.
Define $\phi_{0,1}:\R \times [0,2\pi)\to\C$ and 
$\phi_{m,m+1}:[0,2\pi)\times\R \to\C$ by 
$\phi_{0,1}(a,b)=e^{iab}$ and $\phi_{m,m+1}(a,b)=e^{- iab}$.
For $r=1,\ldots,m$, write $x^r=\{e^{ia^r_1},\ldots,e^{ia^r_n}\}$, where $a^r\in D_n$,
and set $a^0_j=a^{m+1}_j=j-1$, for $j=1,\ldots,n$.
\begin{thm} \label{prod} For $q\ne 1$,
$$\a_m(d\bar{x})= Z_m^{-1}
 \prod_{r=0}^m \det\left(\phi_{r,r+1}(a^r_j,a^{r+1}_k)\right)_{1\le j,k\le n} d\bar{x},$$
 where $Z_m= \tilde c_q^{m-1}(1-(-1)^{n-1}q^{2\pi})^{-(m-1)}(2\pi)^n$.
\end{thm}
\begin{proof}
By (\ref{mi}) we can write
$$\a_m(d\bar{x}) = \tilde c_q^{-(m-1)}(2\pi)^{-n}\Delta(x^1)\Delta(x^m) 
\I_q(x^1,x^2)\cdots \I_q(x^{m-1},x^m)d\bar{x} .$$
Using the formula
$$\Delta(x^1)\Delta(x^m)=\det\left( e^{i(j-(n+1)/2)a^1_k} \right)_{1\le j,k\le n} 
\det\left( e^{-i(j-(n+1)/2)a^m_k} \right)_{1\le j,k\le n} ,$$
and Lemma~\ref{iform}, we obtain
\begin{eqnarray*}
\a_m(d\bar{x}) &=& Z_m^{-1} 
e^{ i \frac{n-1}{2}\sum_j(a^1_j-a^m_j)} 
\det\left( e^{i(j-(n+1)/2)a^1_k} \right)_{1\le j,k\le n} \\
 &\times & \det\left( e^{-i(j-(n+1)/2)a^{m+1}_k} \right)_{1\le j,k\le n}
 \prod_{r=1}^{m-1} \det\left(\phi_{r,r+1}(a^r_j,a^{r+1}_k)\right)_{1\le j,k\le n} d\bar{x} \\
 &=& {Z}_m^{-1} \det\left( e^{i(j-1)a^1_k} \right)_{1\le j,k\le n}
 \det\left( e^{-i(j-1)a^{m+1}_k} \right)_{1\le j,k\le n} \\
 && \times
  \prod_{r=1}^{m-1} \det\left(\phi_{r,r+1}(a^r_j,a^{r+1}_k)\right)_{1\le j,k\le n} d\bar{x} \\
  &=&  {Z}_m^{-1} 
  \prod_{r=0}^{m} \det\left(\phi_{r,r+1}(a^r_j,a^{r+1}_k)\right)_{1\le j,k\le n} d\bar{x} ,
\end{eqnarray*}
as required.
\end{proof}
\begin{cor}
For any $q>0$, the measure $\a$ defines a determinantal point process on
$[0,2\pi)^\Z$
with space-time correlation kernel given by
$$K(r,a;s,b)=\begin{cases}
\frac{1}{2\pi} \sum_{k=0}^{n-1} g_k^{r-s} e^{ i(b-a) k } & r\geq s \\
- \frac{1}{2\pi} \sum_{k \in \Z \setminus \{0,\ldots,n-1\} }
g_k^{r-s} e^{ i(b-a) k }& r<s \end{cases}$$
where
$$
g_k = \left( \int_0^{2\pi} f(u) e^{ - iuk } du \right)^{-1}
=   i \left( k - \frac{n-1}2 \right)
- \log q .
$$
\end{cor}
\begin{proof}
For $q\ne 1$, this follows from Theorem~\ref{prod},
\cite[Proposition 2.13]{j} and a straightforward computation.
The case $q=1$ is obtained by continuity.
\end{proof}

Lemma~\ref{iform} can be used to give a direct proof of~(\ref{mi}),
and hence Proposition~\ref{br}.
\begin{proof}[Proof of Proposition \ref{br}]  
The characters $\chi_\lambda$ are given, for $y=\{e^{ia_1},\ldots,e^{ia_n}\}$, by
$$\chi_\lambda(y)=i^{- {n\choose 2}} \Delta(y)^{-1} \det\left( e^{i\mu_j a_k}\right)_{1\le j,k\le n}  ,$$
where $\mu=\lambda+\rho$ and 
$$\rho=\left(\frac{n-1}{2},\frac{n-1}{2}-1,\ldots,-\frac{n-1}{2}+1,-\frac{n-1}{2}\right).$$
Using Lemma~\ref{iform} and the Cauchy-Binet formula, we obtain
$$\int \frac{\Delta(z)}{\Delta(y)}\I_q(y,z) \chi_\lambda(z) dz = 
(1-(-1)^{n-1}q^{2\pi}) \prod_j (-i\mu_j-\log q)^{-1} \chi_\lambda(y) .$$
On the other hand, writing $x_r=\{e^{ir},1,1,\ldots,1\}$, an easy calculation shows that
$$ \int_0^{2\pi} |1-e^{ir}| q^r \frac{\chi_\lambda(x_r)}{d_\lambda} dr 
= (n-1)! (1-(-1)^{n-1}q^{2\pi}) \prod_j (-i\mu_j-\log q)^{-1}$$
and so, by (\ref{efun}),
$$m_q \chi_\lambda = \tilde c_q^{-1}(1-(-1)^{n-1}q^{2\pi}) \prod_j (-i\mu_j-\log q)^{-1} \chi_\lambda.$$
Since $\{\chi_\lambda, \lambda\in\Omega_n\}$ is a basis for $L_2(C_n,\mu)$, this implies (\ref{mi}).
\end{proof}

Analogous results to those presented in this section
can be obtained for the discrete version of this model,
which is equivalent to considering a certain family of Gibbs measures on 
rhombic tilings of the cylinder.  For more details, see~\cite{am}. 
The couplings defined in Section 3 are quite useful in this setting, where the
group-theoretic considerations of Section 2 no longer apply.
For example, they can be used to prove that the discrete analogues
of the interlacing operators $\{\I_q,q>0\}$ commute with each other.
In this setting, the symmetric functions 
$$(q_1,\ldots,q_k)\mapsto\I_{q_1}\cdots\I_{q_k}(y,z)$$
are essentially the cylindrical skew Schur functions discussed in
the papers~\cite{mcn01,pos05}. 

Finally, we remark that, in the case $q=1$, the probability measure $\alpha_2$ defined 
by~(\ref{alpha}) also arises naturally in random matrix theory.  
The probability measures on $C_n$ given by
$\mu(dx)=(2\pi)^{-n}\Delta(x)^2dx$ and $A_n^{-1}\Delta(x)dx$, where $A_n$ is a normalisation 
constant, are known, respectively, as the {\em circular unitary ensemble} and {\em circular orthogonal 
ensemble}. It is a classical result, which was conjectured by Dyson~\cite{dys62a} and subsequently 
proved by Gunson~\cite{gun62}, that the set of alternate eigenvalues from a superposition of two 
independent draws from the circular orthogonal ensemble, are distributed according to the circular 
unitary ensemble. Moreover, the joint law of the `even' and `odd' eigenvalues has probability
density on $C_n\times C_n$ proportional to $\Delta(y)\Delta(z)\I_1(y,z)$, which is the same as
$\alpha_2$, the joint distribution at two consecutive times of the stationary Markov chain with 
transition kernel $m_q$ (defined by (\ref{mdef})) in the case q=1.

\bigskip

\noindent {\em Acknowledgements.}
Research of the first two authors supported in part by Science Foundation Ireland
Grant No. SFI04/RP1/I512. Thanks are due to Dominique Bakry, 
Manon Defosseux and Persi Diaconis for helpful discussions,
and also to the anonymous referee and Associate Editor for helpful
comments and suggestions which have led to an improved version of the paper.

\end{document}